\documentclass[11pt]{amsart}
\usepackage{amsmath, amssymb, tensor, amsthm, amsfonts, marginnote, tikz-cd, mathtools, xcolor}
\usepackage[hidelinks]{hyperref}
\def\bC{\mathbf{C}}

\def\bR{\mathbf{R}}

\def\bZ{\mathbf{Z}}

\newcommand{\beq}{\begin{equation}}
\newcommand{\eeq}{\end{equation}}

\theoremstyle{plain}
\newtheorem{theorem}[subsubsection]{Theorem}
\newtheorem{proposition}[subsubsection]{Proposition}
\newtheorem{lemma}[subsubsection]{Lemma}

\theoremstyle{remark}

\theoremstyle{definition}

\DeclareMathOperator{\image}{im}

\DeclareMathOperator{\POrth}{PO}

\DeclareMathOperator{\PGL}{PGL}
\DeclareMathOperator{\Hom}{Hom}
\DeclareMathOperator{\Ext}{Ext}

\DeclareMathOperator{\Perv}{Perv}

\DeclareMathOperator{\SL}{SL}

\DeclareMathOperator{\Gr}{Gr}
\DeclareMathOperator{\ch}{char}
\DeclareMathOperator{\id}{id}
\DeclareMathOperator{\soc}{soc}
\DeclareMathOperator{\Rep}{Rep}
\DeclareMathOperator{\stab}{stab}

\DeclareMathOperator{\odd}{odd}
\DeclareMathOperator{\even}{even}

\DeclareMathOperator{\quan}{qnt}

\swapnumbers \numberwithin{equation}{section}

\title{The real affine Grassmannian and quantum SL(2)}
\author{Mark Macerato and Jeremy Taylor}

\begin{document}

\begin{abstract}
We prove that the category of equivariant perverse sheaves on the affine Grassmannian of $\PGL(2, \bR)$ is highest weight and we construct the projective objects. Moreover we prove that the category of perverse sheaves on the odd component is equivalent to the principal block of Lusztig's quantum $\SL(2)$ at a primitive fourth root of unity.
\end{abstract}

\maketitle

\section{Introduction}

\subsection{The real affine Grassmannian}
Let $G \coloneqq \PGL(2, \bR)$ and let $N \subset B$ be the unipotent radical of the Borel. Set $O \coloneqq \bR[[t]]$ and $K \coloneqq \bR((t))$. Let $\Gr \coloneqq G_K/G_O$ be the real affine Grassmannian. The $G_O$-orbits $i_n: \Gr^n \coloneqq G_Ot^n \hookrightarrow \Gr$ are indexed by nonnegative integers. There are two connected components $\Gr^{\even}$ and $\Gr^{\odd}$ containing the even and odd $G_O$-orbits respectively.

Fix a coefficient field $k$ of characteristic 0. Let $\Perv_{G_O}(\Gr)$ be the category of $G_O$-equivariant perverse sheaves (see section \ref{RealPerverse} for conventions on half-integer shifts). This real Satake category is not semisimple, in contrast to \cite{MV}.

The stabilizer in $G$ of $t^n \in \Gr$ is either $B$ if $n > 0$ or else $G$ if $n = 0$. In both cases its component group is $\pi_0(\stab_G(t^n)) = \bZ/2$.
Therefore $\Gr^n$ admits two perverse $G_O$-equivariant local systems.

Denote the standard, middle, and costandard extension of the trivial representation $k^+ \in \Rep(\bZ/2)$ by \[\Delta(n)^+ \coloneqq i_{n!} k^+_{\Gr^n}[n/2], \quad L(n)^+ \coloneqq i_{n!*} k^+_{\Gr^n}[n/2], \quad \nabla(n)^+ \coloneqq i_{n*} k^+_{\Gr^n}[n/2].\]

Denote the standard, middle, and costandard extension of the sign representation $k^- \in \Rep(\bZ/2)$ by \[\Delta(n)^- \coloneqq i_{n!} k^-_{\Gr^n}[n/2], \quad L(n)^- \coloneqq i_{n!*} k^-_{\Gr^n}[n/2], \quad \nabla(n)^- \coloneqq i_{n*} k^-_{\Gr^n}[n/2].\]

We prove that $\Delta(n)^{\pm}$ and $\nabla(n)^{\pm}$ are perverse, hence $\Perv_{G_O}(\Gr)$ is a highest weight category. Moreover we explicitly construct the indecomposable projectives. 

On the odd component, \[\Perv_{G_O}(\Gr^{\odd}) \simeq \Perv_{G_O}(\Gr^{\odd})^+ \oplus \Perv_{G_O}(\Gr^{\odd})^-\] splits as a sum of two blocks. Here $\Perv_{G_O}(\Gr^{\odd})^{\pm}$ are the full subcategories of perverse sheaves whose Jordan-Holder factors are all of the form $L(n)^{\pm}$ respectively.

\subsection{Comparison with quantum SL(2)}
Let $\Rep_i(\SL(2))$ be the category of (finite dimensional, type 1) representations of Lusztig's quantum $\SL(2)$ (defined in \cite{Lus88}) specialized at a primitive fourth root of unity. Let $\Rep_i(\SL(2))^{\even}$ be the principal block.

By constructing indecomposable projective generators on both sides, then explicitly calculating all maps between them, we prove an equivalence of categories \beq \label{RealQuantum}\Perv_{G_O}(\Gr^{\odd})^+ \simeq \Rep_i(\SL(2))^{\even}.\eeq
The $\Rep(\SL(2))$-action on $\Perv_{G_O}(\Gr^{\odd})^+$ via Nadler's nearby cycles functor corresponds to the action on $\Rep_i(\SL(2))^{\even}$ via quantum Frobenius pullback.

\subsection{Connection to symmetric varieties}
Let $X \coloneqq \PGL(2, \bC)/\POrth(2, \bC)$ be the complex symmetric variety corresponding to the real group $G$. The real symmetric equivalence of \cite{CN24} says $\Perv_{G_O}(\Gr) \simeq \Perv_{G_{O_{\bC}}}(X_{K_{\bC}})$, where $O_{\bC} \coloneqq \bC[[t]]$ and $K_{\bC} \coloneqq \bC((t))$. This category would appear in the relative local Langlands conjecture of \cite{BZSV}, were it not excluded for containing a type N root.

\subsection{Acknowledgements}
We are very grateful to Tsao-Hsien Chen for pointing out an important sign error in an earlier draft of this paper. This version contains the same methods and ideas, but different results.

This paper was inspired by Gurbir Dhillon and Jonathan Wang's suggestion that loop spaces of certain symmetric varities are related to quantum groups at $q = i$. We also thank David Nadler for helpful discussions. M.M. and J.T. were partially supported by NSF grant DMS-1646385. M.M. was partially supported by an NSF graduate research fellowship.

\section{The real character functor}
The real character functor is defined in section 8.5 of \cite{Nad} using semi-infinite orbits. In contrast to theorem 3.6 of \cite{MV}, it does not compute cohomology. Here we calculate some values of the real character functor.

\subsection{Perverse sheaves}\label{RealPerverse}
On each connected component of $\Gr$, the dimensions of $G_O$-orbits have the same parity. Therefore it is possible to define the abelian category of perverse sheaves.
Since the $G_O$-orbits in $\Gr^{\odd}$ have odd dimension, it is convenient to allow half-integer shifts. (The differential still increases degree by 1.)

Let $\Perv_{G_O}(\Gr)$ be the category of $G_O$-equivariant half-integer graded complexes of sheaves $F$ such that $i_n^* F$ and $i_n^! F$ are finite rank local systems on $\Gr^n$ concentrated in degrees $\leq -n/2$ and $\geq -n/2$ respectively.

\subsection{Semi-infinite orbits}
Let $S^m \coloneqq N_K t^m \subset \Gr$ and  $T^m \coloneqq N^-_K t^m \subset \Gr$ be the semi-infinite orbits.

\begin{proposition}\label{Intersect}
The only nonempty intersections between semi-infinite and spherical orbits are
\begin{enumerate}
\item[-] $S^m \cap \Gr^n = \begin{cases} \bR^n & m = n \\ \bR^d - \bR^{d-1} & n + m = 2d \text{ for } 0 < d < n  \\ 
\bR^0 & m = -n,\end{cases}$
\item[-] $T^m \cap \Gr^n = \begin{cases} \bR^0 & m = n \\ \bR^d - \bR^{d-1} & n = m + 2d \text{ for } 0 < d < n  \\ 
\bR^n & m = -n,\end{cases}$ 
\end{enumerate}
and $\big(\begin{smallmatrix} 1 & \\ & -1 \end{smallmatrix}\big) \in G$ acts by multiplication by $-1$.
\end{proposition}
\begin{proof}
If $n = m + 2d$ for $0 < d < n$ then \[T^m \cap \Gr^n \; = \; \left\{ \begin{pmatrix} t^m & \\ a_{-d}t^{-d} + \dots a_{-1}t^{-1} & 1 \end{pmatrix} \text{ such that } a_{-d} \neq 0 \right\} \; = \; \bR^d - \bR^{d-1}\] and moreover the action \[\begin{pmatrix} 1 & \\  & -1 \end{pmatrix}\begin{pmatrix} t^m & \\ a_{-d}t^{-d} + \dots a_{-1}t^{-1} & 1 \end{pmatrix} \begin{pmatrix} 1 & \\ & -1 \end{pmatrix} = \begin{pmatrix} t^m & \\ -a_{-d}t^{-d} - \dots a_{-1}t^{-1} & 1 \end{pmatrix}\] is multiplication by $-1$ on $\bR^d - \bR^{d-1}$. The other cases are similar, see lemma A.6.1 of \cite{GR} for more details.
\end{proof}

\subsection{The character functor}
Consider the character functor \[\ch: \Perv_{G_O}(\Gr) \rightarrow \Rep(\bZ/2)_{\bZ}, \quad F \; \mapsto \; \bigoplus_{m \in \bZ} \Gamma_{T^m}(F)[m/2](-m) \simeq \bigoplus_{m \in \bZ} \Gamma_c(S^m, F)[m/2](-m)\] taking values in $\bZ$-graded representations of $\pi_0(T) = \bZ/2$. Here $(-m)$ denotes the $\bZ$-grading. 

Corollary 8.5.1 of \cite{Nad} says that $\ch$ is exact. Moreover $\ch$ is faithful: if $\Gr^n$ is open in the support of $F \in \Perv_{G_O}(\Gr)$ then $\ch F$ is nonzero in weight $n$.

\begin{proposition}\label{CharStar}
The sheaves $\Delta(n)^{\pm}$ and $\nabla(n)^{\pm}$ are perverse. Moreover their characters are \beq \label{CharStarEq} \ch(\Delta(n)^+) \simeq \ch(\nabla(n)^+) \simeq k^+(n) \oplus (k^+ \oplus k^-)(n-2) \oplus \cdots (k^+ \oplus k^-)(-n+2) \oplus (k^-)^{\otimes n}(-n).\eeq The character of $\Delta(n)^-$ and $\nabla(n)^-$ is given by tensoring with the sign representation.
\end{proposition}
\begin{proof}
By base change and proposition \ref{Intersect} imply \[\ch (\Delta(n)^+) \simeq  \bigoplus_{m \in \bZ} H^*_c(S^m \cap \Gr^n)[(-m - n)/2](m) \;\; \text{and} \;\;  \ch (\nabla(n)^+) \simeq \bigoplus_{m \in \bZ} H^*(T^m \cap \Gr^n)(m),\] which matches \eqref{CharStarEq}.

There is an exact triangle \[\tau^{\leq 0} \nabla(n)^{\pm} \rightarrow \nabla(n)^{\pm} \rightarrow \tau^{\geq 1}\nabla(n)^{\pm}\] and $\tau^{\leq 0} \nabla(n)^{\pm}$ is perverse. Therefore $\ch(\tau^{\geq 1} \nabla(n)^{\pm}) \simeq 0$, because \eqref{CharStarEq} is concentrated in degree 0 and $\ch$ is exact. Since $\ch$ is faithful, $\tau^{\leq 0} \nabla(n)^{\pm} \simeq \nabla(n)^{\pm}$ is perverse.
\end{proof}

\subsection{Nadler's functor and even simples}
According to sections 5 and 6 of \cite{Nad}, nearby cycles from the complex affine Grassmannian gives a monoidal functor \beq \label{Complex}\psi: \Perv_{G_{O_{\bC}}}(\Gr_{\bC}) \rightarrow \Perv_{G_O}(\Gr), \qquad L(n)_{\bC} \mapsto L(2n)^+,\eeq exact because $G$ is quasi-split, and sending simples to simples because $G$ is split.

\begin{proposition}\label{CharEven}
The characters of even simples are
\[\ch(L(2n)^{\pm}) \simeq k^{\pm}(2n) \oplus k^{\pm}(2n-4) \oplus \cdots k^{\pm}(-2n).\]
\end{proposition}
\begin{proof}
Theorem 8.8.1 of \cite{Nad} says that $\psi$ preserves characters but doubles the weights. Moreover the component group $\pi_0(T) = \bZ/2$ acts trivially on $\ch(L(2n)^+)$ because the specialization diagram on page 42 of \cite{Nad} is $T$-equivariant.
\end{proof}

By monoidality of Nadler's functor and the Clebsch-Gordan rule, \beq \label{ClebschGordan} L(2n)^+* L(2m)^+ \simeq L(2(n + m))^+ \oplus L(2(n + m) - 4)^+ \oplus \cdots L(2|n - m|)^+.\eeq


\section{Convolution}
Here prove that the real character functor is monoidal. The proof of proposition 6.4 of \cite{MV} fails in the real setting because the Mirkovic-Vilonen spectral sequence no longer degenerates on the first page. Instead we interpret convolution as a fusion product, then use that specialization commutes with hyperbolic localization.

\subsection{Convolution}
If $F, F' \in \Perv_{G_O}(\Gr)$, define $F * F' \coloneqq m_! (F \widetilde{\boxtimes} F')$. Here \beq \label{Conv} m: \Gr \widetilde{\times} \Gr \coloneqq G_K \times_{G_O} \Gr \rightarrow \Gr\eeq is the convolution map, and $F \widetilde{\boxtimes} F'$ is defined as in equation (4.2) of \cite{MV}.
Section 3.8 of \cite{Nad} explains that $m$ is semismall, hence convolution is exact. The abelian categories $\Perv_{G_O}(\Gr)$ and $\Perv_{G_O}(\Gr^{\even})^0$ (but not $\Perv_{G_O}(\Gr)^0$) are preserved by convolution.

\subsection{Specialization and hyperbolic localization}Here we prove that specialization commutes with hyperbolic localization, a real analytic version of \cite{Ric}.

Let $\bR^{\times}$ act on a real analytic space $Y$. Let $f: Y \rightarrow \bR$ be an $\bR^{\times}$-invariant function, $j: Y^{>0} \hookrightarrow Y$ be the inclusion of the locus $f > 0$, and $i: Y^0 \hookrightarrow Y$ be the inclusion of the locus $f = 0$. If $F$ is a sheaf on $Y^{>0}$, define its specialization to $Y^0$ by the formula $\nu F \coloneqq i^*j_*F$.

Let $M$ be a connected component of the fixed locus. 
Let $Y \xhookleftarrow{s} S \xrightarrow{p} M$ be the inclusion of the attracting locus and retraction to the fixed points. Let $s_0$, $p_0$ be the restrictions to the locus $f=0$.
Let $Y \xhookleftarrow{t} T \xrightarrow{q} M$ be the inclusion of the repelling locus and retraction to the fixed points.

In the following lemma, Braden's theorem allows us to commute hyperbolic localization past all other sheaf functors.

\begin{lemma}\label{Hyp}
Let $F$ be an $\bR^{\times}$-constructible sheaf on $Y^{>0}$. Then $(p_0)_!(s_0)^*\nu F \simeq \nu p_! s^* F$.
\end{lemma}
\begin{proof}
Braden's theorem \cite{Bra} says $p_!s^* \simeq  q_*t^!$ for $\bR^{\times}$-constructible sheaves. Therefore by base change \[(p_0)_!(s_0)^*\nu F  \simeq (p_0)_!(s_0)^* i^* j_* F  \simeq (i|_M)^*p_!s^* j_*F  \simeq (i|_M)^*q_*t^! j_*F \simeq (i|_M)^*(j|_M)_*  q_* t^! F  \simeq \nu p_! s^* F.\]
\end{proof}

\subsection{Fusion}
Here we interpret convolution as a fusion product. The ingredients are smoothness of \eqref{GlobalConvOrbits} and commutation of specialization with proper pushforward.

First we define three moduli spaces mapping to $\bR$, with the following special and generic fibers.

\begin{center}
\begin{tabular}{l l l} 
Moduli space & Special fiber & Generic fiber\\
$\Gr_{\bR}$ & $\Gr$ & $\Gr$ \\
$\Gr_{0 \times \bR}$ & $\Gr$ & $\Gr \times \Gr$ \\
$\Gr \widetilde{\times} \Gr_{\bR}$ & $\Gr \widetilde{\times} \Gr$ & $\Gr \times \Gr$ \\
\end{tabular}
\end{center}

\begin{enumerate}
\item Let $\Gr_{\bR}$ parameterize: a point $x \in \bR$, a $G$-bundle $E$ on $\bR$, a trivialization of $E|_{\bR - x}$.

\item Let $\Gr_{0 \times \bR}$ parameterize: a point $x \in \bR$, a $G$-bundle $E$ on $\bR$, a trivialization of $E|_{\bR - \{0, x\}}$.

\item Let $\Gr \widetilde{\times} \Gr_{\bR}$ parameterize: a point $x \in \bR$, $G$-bundles $E, E'$ on $\bR$, a trivialization of $E'|_{\bR - 0}$, an identification $E|_{\bR - x} \simeq E'|_{\bR - 0}$.
\end{enumerate}

Forgetting $E'$ gives the global convolution map $m: \Gr \widetilde{\times} \Gr_{\bR} \rightarrow \Gr_{0 \times \bR}$ whose restriction to $x = 0$ recovers \eqref{Conv}.

Choosing a global coordinate gives an identification $\Gr_{\bR} \simeq \Gr \times \bR$. For $F \in \Perv_{G_O}(\Gr)$, let $F_{\bR}$ be its pullback along the projection $\Gr_{\bR} \rightarrow \Gr$. 

For $F, F' \in \Perv_{G_O}(\Gr)$, let $F \widetilde{\boxtimes} F'_{\bR}$ on $\Gr \widetilde{\times} \Gr_{\bR}$ be defined as in equation (5.6) of \cite{MV}.

\begin{proposition}\label{Fusion}
If $F, F' \in \Perv_{G_O}(\Gr)$ then $F*F' \simeq \nu (F \boxtimes F'_{\bR^{>0}})$.
\end{proposition}
\begin{proof}
The sheaf $F \widetilde{\boxtimes} F'_{\bR}$ on $\Gr \widetilde{\times} \Gr_{\bR}$ is constructible with respect to the stratification by spherical orbits $\Gr^n \widetilde{\times} \Gr_{\bR}^m$. By smoothness of \beq \label{GlobalConvOrbits} \Gr^n \widetilde{\times} \Gr_{\bR}^m \rightarrow \bR,\eeq specialization equals restriction \[\nu (F \widetilde{\boxtimes} F'_{\bR^{>0}}) \simeq (F \widetilde{\boxtimes} F'_{\bR})|_{x = 0} \simeq F \widetilde{\boxtimes} F'.\] 

The global convolution map restricts to an isomorphism on the $x > 0$ locus \[m: \Gr \widetilde{\times} \Gr_{\bR^{>0}} \xrightarrow{\sim} \Gr_{0 \times \bR^{>0}}.\] Moreover specialization commutes with proper pushforward, hence \[\nu(F \boxtimes F'_{\bR^{>0}}) \simeq \nu (m_* (F \widetilde{\boxtimes} F'_{\bR^{>0}})) \simeq m_* \nu(F \widetilde{\boxtimes} F'_{\bR^{>0}}) \simeq m_*(F \widetilde{\boxtimes} F') \simeq F*F'.\]
\end{proof}

\subsection{Monoidality of the character functor}
Theorem 8.9.2 of \cite{Nad} says that $\ch$ is monoidal on the essential image of \eqref{Complex}. Here we prove the stronger result that $\ch$ is monoidal on the entire category $\Perv_{G_O}(\Gr)$.

As in the proof of proposition 6.4 of \cite{MV}, consider the global semi-infinite orbits \beq \label{GlobalCT} \Gr_{0 \times \bR} \xhookleftarrow{s} (\Gr_B)_{0 \times \bR} \xrightarrow{p} (\Gr_T)_{0 \times \bR} \quad \text{and} \quad \Gr_{0 \times \bR} \xhookleftarrow{t} (\Gr_{B^-})_{0 \times \bR} \xrightarrow{q} (\Gr_T)_{0 \times \bR}.\eeq

\begin{proposition}
If $F, F' \in \Perv_{G_O}(\Gr)$ then \beq \label{CharMon}\ch (F * F') \simeq (\ch F) \otimes (\ch F').\eeq
\end{proposition}
\begin{proof}
Proposition \ref{Fusion} and lemma \ref{Hyp} imply $(p_0)_! (s_0)^* (F * F') \simeq \nu p_! s^* (F \boxtimes F'_{\bR^{>0}})$, an isomorphism of sheaves on $\Gr_T = \bZ$. Taking stalks at each connected component gives \eqref{CharMon}, an isomorphism of $\bZ/2$-representations because \eqref{GlobalCT} is $T$-equivariant.
\end{proof}

\subsection{Duality}
Let $H$ be a group ind-proper ind-real analytic space. Let $a: H \rightarrow H$ be the inversion map $a(h) = h^{-1}$. Convolution is defined $F * F' \coloneqq m_!(p_1^* F \otimes p_2^* F)$, where $p_1, p_2, m: H \times H \rightarrow H$ are the two projections and the multiplication map. Let $D$ denote Verdier duality.

\begin{proposition}\label{GroupDuality}
If $F \in \Perv(H)$, there is an adjunction \[\Hom(F*-, -) \simeq \Hom(-, (Da^*F)*-).\]
\end{proposition}
\begin{proof}
Using the change of coordinates \[H \times H \rightarrow H \times H, \qquad (h_1, h_2) \mapsto (h_1^{-1}, h_1h_2),\]  proposition 2.9.4 of \cite{Ach}, and ind-properness of $m$, we get \begin{align*}\Hom(F*-, -) &\simeq \Hom((p_1^* F) \otimes (p_2^* -), m^! -) \\ &\simeq \Hom((p_1^* a^* F) \otimes (m^* -), p_2^! -) \\ &\simeq \Hom(m^* -, (p_1^* (Da^*F)) \otimes (p_2^* -)) \\ &\simeq \Hom(-, (Da^*F) * -).\end{align*}
\end{proof}

If $F \in \Perv_{G_O}(\Gr)$, let $a^* F$ be the pullback along inversion \[a: G_K \rightarrow G_K, \qquad g \mapsto g^{-1}.\]

\begin{proposition}\label{Duality}
For $F \in \Perv_{G_O}(\Gr)$, there is an adjunction \[\Hom(F * -, -) \simeq \Hom(-, (Da^*F)*-).\]
\end{proposition}
\begin{proof}
Let $\Omega_{\bC} \subset G_{\bC[t^{\pm 1}]}$ be the group of polynomial loops $(S^1, 1) \rightarrow (K, 1)$, into the maximal compact subgroup $K \subset G_{\bC}$. Gram-Schmidt factorization $G_K \simeq \Omega_{\bC} \times G_O$ implies a homeomorphism $\Gr_{\bC} \simeq \Omega_{\bC}$. Moreover the real affine Grassmannian $\Gr \subset \Gr_{\bC}$ identifies with a real subgroup $\Omega \coloneqq \Omega_{\bC} \cap G_{\bR[t^{\pm 1}]}$. The convolution map \eqref{Conv} identifies with multiplication $m: \Omega \times \Omega \rightarrow \Omega$, see \cite{CN18}. Proposition \ref{GroupDuality} now implies the desired adjunction. 
\end{proof}

\section{The real Satake category}\label{RealSat}

\subsection{The odd simples}
The odd simples are obtained from the even simples by convolution with $L(1)^+$.

\begin{proposition}\label{Steinberg}
There is an isomorphism $L(2n+1)^{\pm} \simeq L(1)^+ * L(2n)^{\pm} \simeq L(2n)^{\pm} * L(1)^+$.
\end{proposition}
\begin{proof}
Since the character functor is monoidal, \beq \label{CharOdd} \ch(L(1)^+ * L(2n)^+) \simeq k(2n+1)^+ \oplus k(2n-1)^- \oplus k(2n-3)^+ \oplus \cdots  k(-2n-1)^-.\eeq For support reasons $L(2n+1)^+$ appears as a composition factor in $L(1)^+*L(2n)^+$. 
\begin{enumerate}
\item Suppose that $L(1)^{\pm}$ and $L(2n+1)^+$ are the only Jordan-Holder factors in $L(1)^+*L(2n)^+$. Then \[\ch L(2n+1)^+ \simeq k(2n+1)^+ \oplus k(2n-1)^- \oplus \cdots k(3)^{\pm} \oplus k(-3)^{\mp} \oplus k(-5)^{\pm} \oplus \cdots,\] hence $\ch L(2n+1)^+ * L(1)^+$ vanishes in weight 0 but not in weight 4. This contradicts proposition \ref{CharEven} and exactness of $\ch$.
\item Suppose that $L(1)^+*L(2n)^+$ contains a Jordan-Holder factor $L(2m+1)^{\pm}$ for some $1< m < n$. By induction $\ch L(2m+1)^{\pm}$ is nonzero in weights $\pm 1$ and $\pm 3$. Therefore $\ch L(2n+1)^+$ vanishes in weights $\pm 1$ and $\pm 3$. Hence $\ch(L(2n+1)^+ * L(1)^+)$ is nonzero but vanishes in weights 0 and $\pm 2$. This contradicts proposition \ref{CharEven} and exactness of $\ch$.
\end{enumerate}
Therefore $L(1)^+*L(2n)^+ \simeq L(2n+1)^+$ is simple.
\end{proof}

The following lemma gives the Jordan-Holder filtration for costandards. If $n = 0$ we set $L(-1)^{\pm} \coloneqq 0$.

\begin{lemma}\label{Odd*}
There are nonsplit short exact sequences
\begin{enumerate}
\item \label{SESDelta} $0 \rightarrow L(2n-1)^{\pm} \rightarrow \Delta(2n+1)^{\pm} \rightarrow L(2n+1)^{\pm} \rightarrow 0$,
\item \label{SESNabla} $0 \rightarrow L(2n+1)^{\pm} \rightarrow \nabla(2n+1)^{\pm} \rightarrow L(2n-1)^{\pm} \rightarrow 0.$ 
\end{enumerate}
\end{lemma}
\begin{proof}
The signed characters of the simples $L(n)^{\pm}$ are linearly independent by proposition \ref{CharEven} and equation \eqref{CharOdd}. Therefore the simple constituents of $\nabla(2n+1)^{\pm}$ can be determined from its character \eqref{CharStarEq} and proposition \ref{Steinberg}. The short exact sequence \eqref{SESNabla} does not split  otherwise the costalk of $\nabla(2n+1)^+$ along $\Gr^{2n-1}$ would be nonzero. By Verdier duality we also get a nonsplit short exact sequence \eqref{SESDelta}.
\end{proof}

\subsection{Highest weight structure}
We follow section 3.2 of \cite{BGS} for the definition of a highest weight category, except we do not require finitely many simple objects. 

\begin{proposition}
The standard objects $\Delta(n)^{\pm} \hookrightarrow L(n)^{\pm}$ and costandard objects $L(n)^{\pm} \twoheadrightarrow \nabla(n)^{\pm}$ make $\Perv_{G_O}(\Gr)$ a highest weight category.
\end{proposition}
\begin{proof}
Recall that $\Delta(n)^{\pm}$ and $\nabla(m)^{\pm}$ are perverse by proposition \ref{CharStar}.
The only nontrivial condition is that $\Ext^2(\Delta(n)^{\pm}, \nabla(m)^{\pm}) = 0$ which follows by adjunction. 
\end{proof}

By BGG reciprocity, the odd projectives are described as follows.

\begin{proposition}\label{ProjStandard}
The projective cover of $L(2n+1)^{\pm}$ admits a nonsplit standard filtration \beq \label{PStandard}0 \rightarrow \Delta(2n+3)^{\pm} \rightarrow P(2n+1)^{\pm} \rightarrow \Delta(2n+1)^{\pm} \rightarrow 0.\eeq Therefore $P(2n+1)^{\pm}$ admits a Jordan-Holder filtration with graded pieces $L(2n+1)^{\pm}, L(2n+3)^{\pm}, L(2n-1)^{\pm}, L(2n+1)^{\pm}$. (If $n = 0$ we set $L(-1)^{\pm} \coloneqq 0$.)
\end{proposition}
\begin{proof}
The existence of a projective cover $P(2n+1)^{\pm}$ with a standard filtration is a general result about highest weight categories, see theorem 3.2.1 of \cite{BGS}. Since $\Perv_{G_O}(\Gr)$ has infinitely many simple objects, $P(2n+1)^{\pm}$ may a priori be a pro-object. But the following calculation shows that it is a genuine object supported on $\overline{\Gr}^{2n+3}$.

By BGG reciprocity and lemma \ref{Odd*}, the number of times that $\Delta(m)^?$ appears in the standard filtration of $P(2n+1)^+$ is \[\dim \Hom(P(2n+1)^+, \nabla(m)^?) = [L(2n+1)^+: \nabla(m)^?] = \begin{cases} 1 & \text{ if } ? = + \text{ and } m = 2n+1 \text{ or } 2n+3, \\ 0 & \text{ otherwise.}
\end{cases}\] Moreover \eqref{PStandard} does not split because $P(2n+1)^{\pm}$ is indecomposable.
\end{proof}

\subsection{Structure of projectives}
We now explain how all projectives, including on the even component, can be constructed from $P(1)^+$ by convolution.

\begin{proposition}\label{ProjConvolve}
\begin{enumerate}
\item The projective cover of $L(2n)^{\pm}$ is $P(2n)^{\pm} \simeq L(2n+1)^{\pm} * P(1)^+$.
\item The projective cover of $L(2n+1)^{\pm}$ is $P(2n+1)^{\pm} \simeq L(2n)^{\pm}* P(1)^+$. 
\end{enumerate}
\end{proposition}
\begin{proof}
Using proposition \ref{Duality} and exactness of convolution, the functor \[\Hom(L(2n+1)^+*P(1)^+, -) = \Hom(P(1)^+, L(2n+1)^+ * -)\] is exact, hence $L(2n+1)^+ * P(1)^+$ is projective. Moreover $L(2n+1)^+ * L(2m)^? \simeq L(1)^+ * L(2n)^+ * L(2m)^?$ admits $L(1)$ as a Jordan-Holder factor only if $? = +$ and $m = n$ by \eqref{ClebschGordan} and proposition \ref{Steinberg}. Therefore \[\Hom(L(2n+1)^+* P(1)^+, L(m)^?) \simeq \Hom(P(1)^+, L(2n+1)^+ * L(m)^?) \simeq \begin{cases} k & \text{ if } ? = + \text{ and }m = 2n, \\ 0 & \text{ otherwise,}\end{cases}\] hence $L(2n+1)^+ * P(1)^+ \simeq P(2n)^+$ is the projective cover of $L(2n)^+$.
\end{proof}

\begin{proposition}\label{SocProj}
The socles of odd projectives are given by $\soc(P(2n+1)^{\pm}) \simeq L(2n+1)^{\pm}$.
\end{proposition}
\begin{proof}
First consider the case $n = 0$. Then we have a nonsplit short exact sequence \[0 \rightarrow \Delta(3)^{\pm} \rightarrow P(1)^{\pm} \rightarrow L(1)^{\pm} \rightarrow 0.\] If $m \neq 0$ then $\Hom(L(2m+1)^{\pm}, P(1)^+) \simeq \Hom(L(2m+1)^{\pm}, \Delta(3)^+) \simeq 0$  by lemma \ref{Odd*}.

For $n$ general, proposition \ref{ProjConvolve} implies that \[\Hom(L(2m+1)^?, P(2n+1)^+) \simeq \Hom(L(2n)^+ *L(2m)^? * L(1)^+, P(1)^+) \simeq \begin{cases} k & \text{ if } ? = + \text{ and }m = n, \\ 0 & \text{ otherwise.}\end{cases}\]
\end{proof}

\subsection{Blocks on the odd component} Here we prove that the trivial and sign representations of $\bZ/2$ do not interact on the odd component.

Let $\Perv_{G_O}(\Gr^{\odd})^{\pm}$ be the full subcategory of perverse sheaves whose Jordan-Holder factors are all of the form $L(n)^{\pm}$ respectively.

\begin{proposition}
On the odd component \[\Perv_{G_O}(\Gr^{\odd}) \simeq \Perv_{G_O}(\Gr^{\odd})^+ \oplus \Perv_{G_O}(\Gr^{\odd})^-.\] 
\end{proposition}
\begin{proof}
Follows because each simple $L(2n+1)^+$ admits a projective cover $P(2n+1)^+$ whose Jordan-Holder factors are of the form $L(2m+1)^+$. 
\end{proof}

On the even component there is no such splitting. 

\subsection{Maps between projectives on the odd component} We now explicitely calculate all maps between projectives on the odd component.

\begin{proposition}\label{ProjMaps}
On the odd component, all nonzero maps between projectives are spanned by 
\begin{enumerate}
\item[-] $x_n \in \Hom(P(2n+1)^+, P(2n+3)^+)$, 
\item[-] $y_n \in \Hom(P(2n+1)^+, P(2n-1)^+)$, \item[-]  $\id_n, z_n \in \Hom(P(2n+1)^+, P(2n+1)^+)$,
\end{enumerate}
with relations $z_n = x_{n-1}y_n = y_{n+1}x_n$, $z_n^2 = 0$, $x_nz_n = 0$, and $y_nz_n = 0$.
\end{proposition}
\begin{proof}
Proposition \ref{ProjStandard} and \ref{Odd*} imply that $P(2n+1)^+$ has Jordan-Holder factors $L(2n+1)^+, L(2n-1)^+, L(2n+3)^+, L(2n+1)^+$. Therefore \begin{enumerate}
\item[-] $\dim \Hom(P(2n+3)^+, P(2n+1)^+) = \dim \Hom(P(2n-1)^+, P(2n+1)^+) = 1$,
\item[-] $\dim \Hom(P(2n+1)^+, P(2n+1)^+) = 2$,
\end{enumerate}
and all other homs between projectives vanish. 
Therefore $\Hom(P(2n+1)^+, P(2n+1)^+)$ is spanned by the identity and \[z_n: P(2n+1)^+ \twoheadrightarrow \Delta(2n+1)^+ \twoheadrightarrow L(2n+1)^+ \hookrightarrow \Delta(2n+3)^+ \hookrightarrow P(2n+1)^+.\] Up to scaling $z_n \in \Hom(P(2n+1)^+, P(2n+1)^+)$ is the unique nonzero map that vanishes on $\soc P(2n+1)^+ \simeq L(2n+1)^+$.

Both $\image y_n \subset P(2n-1)^+$ and $\image x_{n-1} \subset P(2n+1)^+$ contain $L(2n+1)^+$ as a Jordan-Holder factor by proposition \ref{SocProj}. Therefore $\ker x_{n-1} \subset P(2n-1)^+$ does not contain $L(2n+1)^+$ as a Jordan-Holder factor, hence $\image y_{n} \not\subset \ker x_{2n-1}$. Thus the composition \[P(2n+1)^+ \xrightarrow{y_{n}} P(2n-1)^+ \xrightarrow{x_{n-1}} P(2n+1)^+\] is nonzero. Moreover $x_{n-1}y_{n}$ vanishes on $\soc P(2n+1)^+ \simeq L(2n+1)^+$ because $\Hom(L(2n+1)^+, P(2n-1)^+) \simeq 0$ by proposition \ref{SocProj}. A similar argument shows that $y_{n+1}x_{n}$ is nonzero and vanishes on $L(2n+1)^+ \subset P(2n+1)^+$. Therefore after rescaling we get $z_n = x_{n-1}y_n = y_{n+1}x_n$.

Since $\image z_n \simeq \soc P(2n+1)$ and $z_n, x_n, y_n$ are not injective (for Jordan-Holder factor reasons), it follows that $z_n^2 = 0$, $x_nz_n = 0$ and $y_nz_n = 0$.
\end{proof}

\section{Comparison with quantum SL(2)}
Here we construct all projectives in $\Rep_i(\SL(2))$ and calculate all maps between them, following the same strategy as in section \ref{RealSat}.

\subsection{Quantum SL(2)} The category of representations of Lusztig's quantum group splits into two blocks \[\Rep_i(\SL(2)) \simeq \Rep_i(\SL(2))^{\even} \oplus \Rep_i(\SL(2))^{\odd}\] consisting of representations with only even or odd weights respectively. Denote the Weyl, simple, and dual Weyl modules of highest weight $n$ by \[\Delta(n)^{\quan},\; L(n)^{\quan},\; \nabla(n)^{\quan} \; \in \; \Rep_i(\SL(2)).\] See \cite{And} for a survey of Lusztig's quantum group.

\begin{proposition}\label{QuantumSemisimple}
$\Rep_i(\SL(2))^{\odd}$ is semisimple.
\end{proposition}
\begin{proof}
The Steinberg tensor product theorem (theorem 7.4 of \cite{Lus}) says \[L(2n+1)^{\quan} \simeq L(1)^{\quan} \otimes L(2n)^{\quan},\] and that $L(2n)^{\quan}$ is obtained by quantum Frobenius pullback. Therefore \[\Delta(2n+1)^{\quan} \twoheadrightarrow L(2n+1)^{\quan} \hookrightarrow \nabla(2n+1)^{\quan}\] are isomorphisms because $\Delta(2n+1)^{\quan}, L(2n+1)^{\quan}, \nabla(2n+1)^{\quan}$ have the same character.

The following standard argument shows that \[\Ext^1(L(2n+1)^{\quan}, L(2m + 1)^{\quan}) \simeq \Ext^1(\Delta(2n+1)^{\quan}, \nabla(2m + 1)^{\quan}) \simeq 0.\] The Weyl module $\Delta(2n+1)^{\quan}$ is projective in the Serre subcategory $\Rep_i(\SL(2))^{\leq 2n+1}$ of representations whose weights are at most $2n+1$. If $m \leq n$ then $\Ext^1(\Delta(2n+1)^{\quan}, \nabla(2m+1)^{\quan}) \simeq 0$. If $n < m$ then by duality $\Ext^1(\Delta(2n+1)^{\quan}, \nabla(2m+1)^{\quan}) \simeq \Ext^1(\Delta(2m+1)^{\quan}, \nabla(2n+1)^{\quan}) \simeq 0$.
\end{proof}

\subsection{Projectives}
Using semisimplicity of the odd block, we construct the indecomposable projectives in the even block.

\begin{proposition}\label{ProjInjQuan}
The module $P(2n)^{\quan} \coloneqq  L(2n+1)^{\quan} \otimes L(1)^{\quan}$ is both the projective cover and injective hull of $L(2n)^{\quan}$.
\end{proposition}
\begin{proof}
Because the odd block is semisimple, $\Hom(P(2n)^{\quan}, -) = \Hom(L(1)^{\quan}, L(2n+1)^{\quan} \otimes -)$ is exact. Moreover \[\Hom(P(2n)^{\quan}, L(2m)^{\quan}) = \Hom(L(1)^{\quan}, L(1)^{\quan} \otimes L(2n)^{\quan} \otimes L(2m)^{\quan}) = \begin{cases} k & \text{ if } n = m, \\ 0 & \text{ otherwise,}\end{cases}\] 
Therefore $P(2n)^{\quan}$ is the projective cover of $L(2n)^{\quan}$. A similar argument shows that it is also the injective hull.
\end{proof}

Now we compute all nonzero maps between projectives.

\begin{proposition}\label{ProjMapsQuan}
In the even block, all nonzero maps between projectives are spanned by 
\begin{enumerate}
\item[-] $x_n \in \Hom(P(2n)^{\quan}, P(2n+2)^{\quan})$, 
\item[-] $y_n \in \Hom(P(2n)^{\quan}, P(2n-2)^{\quan})$, \item[-]  $\id_n, z_n \in \Hom(P(2n)^{\quan}, P(2n)^{\quan})$,
\end{enumerate}
with relations $z_n = x_{n-1}y_n = y_{n+1}x_n$, $z_n^2 = 0$, $x_nz_n = 0$ and $y_nz_n = 0$.
\end{proposition}
\begin{proof}
Proposition \ref{ProjInjQuan} implies $\soc P(2n)^{\quan} \simeq L(2n)^{\quan}$. Moreover for character reasons $P(2n)^{\quan}$ admits a Jordan-Holder filtration with graded pieces $L(2n)^{\quan}, L(2n-2)^{\quan}, L(2n+2)^{\quan}, L(2n)^{\quan}$. Therefore the current proposition follows by the same argument as in proposition \ref{ProjMaps}.
\end{proof}

\subsection{The real affine Grassmannian and quantum SL(2)} Now we compare perverse sheaves on the odd component of the real affine Grassmannian to the even block of quantum $\SL(2)$ at a fourth root of unity.

\begin{theorem}
There is an equivalence of abelian categories $\Perv_{G_O}(\Gr^{\odd})^+ \simeq \Rep_i(\SL(2))^{\even}$.
\end{theorem}
\begin{proof}
Follows by proposition \ref{ProjMaps} and \ref{ProjMapsQuan}. Indeed both sides are equivalent to finite dimensional continuous representations of the topological algebra \[\prod_{n, m} \Hom(P(2n+1)^+, P(2m+1)^+) \simeq \prod_{n, m} \Hom(P(2n)^{\quan}, P(2m)^{\quan}),\] i.e. representations such that all but finitely many components of the product act by zero.
\end{proof}

\end{document}